\documentclass[11pt, reqno]{amsart}
\usepackage{fancyhdr}
\usepackage{lipsum}
\usepackage{dsfont}
\usepackage{geometry}
\usepackage{amsmath,amsfonts,amssymb,mathrsfs,enumerate,amsopn,amsthm, graphicx}
\usepackage{hyperref}
\hypersetup{hypertex=true,
colorlinks=true,
linkcolor=black,
anchorcolor=black,
citecolor=black}
\allowdisplaybreaks[4]
\usepackage{enumitem}
\usepackage{color}
\usepackage{bm}
\usepackage{epstopdf}
\usepackage{algorithmicx,enumerate,algorithm}
\usepackage{algpseudocode}
\usepackage{cite}
\usepackage{exscale}
\usepackage{relsize}

\numberwithin{equation}{section}
\newtheorem{theorem}{Theorem}[section]
\newtheorem{definition}{Definition}[section]
\newtheorem{lemma}{Lemma}[section]
\newtheorem{remark}{Remark}[section]

\usepackage{array}
\newlength{\onesixth}
\newtheoremstyle{noparens}%
  {}{}%
  {\itshape}{}%
  {\bfseries}{.}%
  { }%
  {\thmname{#1}\thmnumber{ #2}\mdseries\thmnote{ #3}}

\theoremstyle{noparens}
\newtheorem{lemmaNoParens}[lemma]{Lemma}

\makeatletter \@addtoreset{equation}{section} \makeatother

\begin{document}

\date{\today}

\title{On blow-up to the one-dimensional Navier-Stokes equations with degenerate viscosity and vacuum}

\author{Yue Cao}
\address[Y. Cao]{School of Mathematics, East China University of Science and Technology, Shanghai, 200237, China.}
\email{\tt cao\_yue12@ecust.edu.cn}

\author{Yachun li}
\address[Y. Li]{School of Mathematical Sciences, MOE-LSC, and SHL-MAC, Shanghai Jiao Tong University, Shanghai 200240, P. R. China} \email{\tt ycli@sjtu.edu.cn}

\author{SHAOJUN YU}
\address[S. Yu]{School of Mathematical Sciences, Shanghai Jiao Tong University, Shanghai, 200240, China}\email{\tt edwardsmith123@sjtu.edu.cn}


\date{}
\maketitle

\begin{abstract}
In this paper,  we consider the Cauchy problem of the isentropic compressible Navier-Stokes equations with degenerate viscosity and vacuum in $\mathbb{R}$, where the viscosity depends on the density in a super-linear power law(i.e., $\mu(\rho)=\rho^\delta, \delta>1$). We first obtain  the local existence of the regular solution, then show that the regular solution will blow up in finite time if initial data have an isolated mass group, no matter how small and smooth the initial data are. 
It is worth mentioning that based on the transport structure of some intrinsic variables, we obtain the $L^\infty$ bound of the density, which helps to remove the restriction $\delta\leq \gamma$ in Li-Pan-Zhu\cite{Li2019On} and Huang-Wang-Zhu\cite{Huang2023Singu}.
\end{abstract}


\section{Introduction}
It is well known that the motion of viscous isentropic fluid in $\mathbb{R}$ can be controlled by isentropic Navier-Stokes equations
\begin{equation}\label{1}
\left\{
\begin{aligned}
&\rho_t+(\rho u)_x=0,\\
&(\rho u)_t+(\rho u^2)_x+p(\rho)_x=(\mu(\rho) u_x)_x,
\end{aligned}
\right.
\end{equation}
where $t \geq 0$ is the time variable, $x \in \mathbb{R}$ is the space variable, $\rho(t,x)$ is the mass density, $u(t,x) $ is the fluid velocity, $p$ is the material pressure satisfying the equation of state 
\begin{equation}
p(\rho)=\rho^\gamma
\end{equation}
 for polytropic gas, where $\gamma>1$ is the adiabatic exponent.   $\mu$  is the viscosity giving by
\begin{equation}\label{82}
\mu(\rho)= \rho^\delta, \quad \delta >1.
\end{equation}
In this paper, we consider the Cauchy problem of \eqref{1}
with initial data
\begin{equation}\label{65}
(\rho,u)|_{t=0}=(\rho_0,u_0)(x), \quad x\in \mathbb{R},
\end{equation}
and far field behavior
\begin{equation}\label{44}
(\rho, u) \rightarrow(0,0) \quad \text { as } \quad|x| \rightarrow +\infty, \quad t \geq 0.
\end{equation}

For the constant viscous fluid (i.e., $\delta=0$ in \eqref{82}), there is a lot of literature on the well-posedness of classical solutions. When $\inf_{x} \rho_0(x)>0$, the local well-posedness of  classical solutions  follows from the standard symmetric hyperbolic-parabolic structure satisfying the well-known Kawashima's condition (cf. \cite{Nash1962Le,Kawa1983Sys,   Serrin1959On}), which has been extended to a global one by Matsumura-Nishida \cite{Matsu1980The} near the nonvacuum equilibrium.
 In $\mathbb{R}$, the global well-posedness of strong solutions with arbitrarily large data in some bounded domains has been proven by Kazhikhov-Shelukhin \cite{Kazhikhov1977Unique}, and later, Kawashima-Nishida \cite{Kawashima1981Global} extended this theory to unbounded domains. 
When $\inf_{x} \rho_0(x)=0$, 
 the first main issue is the degeneracy of the time evolution operator, which makes it difficult to catch the behavior of the velocity field near the vacuum, for this case, the local-in-time well-posedness of strong solutions with vacuum was firstly solved by Cho-Choe-Kim\cite{Cho2004Unique} and Cho-Kim\cite{Cho2006On} in $\mathbb{R}^3$, where they introduced an initial compatibility condition to compensate the lack of a positive lower bound of density.
Later, Huang-Li-Xin \cite{Huang2012Glo} extended the local existence to a global one under some initial smallness assumption in $\mathbb{R}^3$. Jiu-Li-Ye \cite{Jiu2014Glo} proved the global existence of classical solution with arbitrarily large data and vacuum in $\mathbb{R}$.

For the degenerate viscous flow (i.e., $\delta> 0$ in \eqref{82}), they have received extensive attentions in recent years, especially for the case with
vacuum, where the well-posedness of classical solutions become more challenging due to the degenerate viscosity. In fact, the high order regularity estimates of the velocity in \cite{Cho2004Unique,Huang2012Glo,Jiu2014Glo} ($\delta=0$) strongly rely on the uniform ellipticity of the Lam\'e operator.
While for $\delta>0$, $\mu(\rho)$ vanish as the density function connects to vacuum continuously, thus it is difficult to adapt the approach of the constant viscosity case. A remarkable discovery of a new mathematical entropy function was made by Bresch-Desjardins \cite{Bre2003Exi}
for the viscosity satisfying some mathematical relation, 
 which provides additional regularity on some derivative of the density.
This observation was applied widely in proving the global existence of weak solutions with vacuum for \eqref{1} and some related models; see Bresch-Desjardins \cite{Bre2003Exi}, Bresch-Vasseur-Yu \cite{Bresch2022Glo}, Jiu-Xin \cite{Jiu2008The}, $\mathrm{Li}$-Xin \cite{Li2015Glo}, Mellet-Vasseur \cite{Mellet2007On}, Vasseur-Yu \cite{Vasseur2016Exi}, and so on. 
Moreover, 
when  $\delta=1$, Li-Pan-Zhu \cite{Li2017On} obtained the local existence of 2-D classical solution with far field vacuum, which also applies to the 2-D shallow water equations. 
When $1 < \delta \leq \min \left\{3, \frac{\gamma+1}{2}\right\}$, 
by making full use of the symmetrical structure of the hyperbolic operator and the weak smoothing effect of the elliptic operator, Li-Pan-Zhu \cite{Li2019On} established
  the local existence of 3-D classical solutions with arbitrarily large data and vacuum, see also Geng-Li-Zhu \cite{Geng2019Vanishing} for more related result, and Xin-Zhu \cite{Xin2021Glo} for the global existence of classical solution under some initial smallness assumptions in homogeneous Sobolev space. 
When $0<\delta<1$, 
Xin-Zhu \cite{Xin2021Well} obtained the local existence of 3-D local classical solution with far field vacuum, Cao-Li-Zhu\cite{Cao2022Glo} proved the global existence of 1-D classical solution with large initial data.   Some other interesting results and discussions can also be seen in Chen-Chen-Zhu\cite{Chen2022Vanishing}, Germain-Lefloch \cite{Germain2016Finite}, Guo-Li-Xin \cite{Guo2012Lagrange}, Lions \cite{Lions1998Math}, Luo-Xin-Zeng \cite{Luo2016Non}, Yang-Zhao \cite{Yang2002Vacuum}, Yang-Zhu \cite{Yang2002Com}, and so on.

It should be noted that, even though much important progress has been obtained for the degenerate  isentropic compressible Navier-Stokes equations,  due to the complicated mathematical structure and lack of smooth effect on solutions when vacuum appears, a lot of fundamental questions remain open, e.g., what kind of initial data may cause the formation of the singularities. We refer to Li-Pan-Zhu\cite{Li2019On} and Huang-Wang-Zhu\cite{Huang2023Singu} for the related  formation of singularities on degenerate viscous fluid.
However, their results rely on the restriction that the power of viscosity coefficient is not bigger than the adiabatic exponent, i.e. $\delta\le \gamma$. In the current paper, we remove this restriction by establishing the uniform $L^\infty$ bound of the density, see Theorem \ref{77} for details.

Now, we first introduce the definition of regular solution to Cauchy problem \eqref{1}-\eqref{44}.
\begin{definition} \label{67}{\rm (Regular solution)} Let constant $T>0$. $(\rho, u)(t,x)$ is called a regular solution to Cauchy problem \eqref{1}-\eqref{44}  in $[0, T] \times \mathbb{R}$ if $(\rho, u)(t,x)$ satisfies this problem in the sense of distributions and
\begin{equation*}
\begin{split}
 &  \rho \geq 0, \quad \rho^{\frac{\delta-1}{2}} \in C\left([0, T] ; H^2\right), \quad \rho^{\frac{\gamma-1}{2}} \in C\left([0, T] ; H^2\right);\\
&  u \in C([0, T] ; H^{s^{\prime}}) \cap L^{\infty}\left([0, T] ; H^2\right), \ \ s^{\prime} \in[1,2), \quad \rho^{\frac{\delta-1}{2}}  u_{xxx} \in L^2\left([0, T]; L^2\right);
 \end{split}
 \end{equation*}
 and $u_t+u u_x=0$ as $\rho=0$.
\end{definition}

Throughout this paper, we adopt the following simplified notations:
\begin{equation*}
\begin{aligned}
&|f|_p  =\|f\|_{L^p\left(\mathbb{R}\right)},\quad \|f\|_s=\|f\|_{H^s\left(\mathbb{R}\right)}, \quad D^{k, r} =\big\{f \in L_{l o c}^1\big(\mathbb{R}\big): |\partial_x^k f|_r<+\infty\big\}, \\
& D^k =D^{k, 2},\quad |f|_{D^{k, r}}=\|f\|_{D^{k, r}\left(\mathbb{R}\right)},\quad |f|_{D^k}=\|f\|_{D^k\left(\mathbb{R}\right)},\\
&\|(f, g)\|_X=\|f\|_X+\|g\|_X, \quad \|f\|_{X\cap Y}=\|f\|_X+\|f\|_Y.
\end{aligned}
\end{equation*}

Before introducing the blow-up results, we have the following local existence of regular solutions to Cauchy problem \eqref{1}-\eqref{44}.

\begin{theorem}\label{38} {\rm (Local existence)} Assume that
\begin{equation}\label{78}
1<\min\{\delta,\gamma \}\leq 3.
\end{equation}
 If initial data $\left(\rho_0, u_0\right)$ satisfies
\begin{equation}\label{66}
\rho_0 \geq 0, \quad\Big(\rho_0^{\frac{\gamma-1}{2}}, \rho_0^{\frac{\delta-1}{2}}, u_0\Big) \in H^2,
\end{equation}
then there exists a time $T_*>0$ and a unique regular solution $(\rho, u)(t,x)$ in $\left[0, T_*\right] \times \mathbb{R}$ to Cauchy problem \eqref{1}-\eqref{44} satisfying
\begin{equation*}
\begin{aligned}
\sup _{0 \leq t \leq T_*} & \big(\|\rho^{\frac{\gamma-1}{2}}\|_2^2+\|\rho^{\frac{\delta-1}{2}}\|_2^2+\|u\|_2^2\big)(t)  +\int_0^t |\rho^{\frac{\delta-1}{2}}  u_{xxx}|_2^2 \mathrm{d} s \leq C^0
\end{aligned}
\end{equation*}
for   positive constant $C^0=C^0( \gamma, \delta, \rho_0, u_0)$. In fact, $(\rho, u)$ satisfies Cauchy problem \eqref{1}-\eqref{44} classically in positive time $\left(0, T_*\right]$.
\end{theorem}

For the proof of Theorem \ref{38}, we refer to Geng-Li-Zhu (see Theorem 1.2 of \cite{Geng2019Vanishing}), where the  local existence of regular solution to 3-D compressible Navier-Stokes equations with degenerate viscosities and vacuum is obtained. The proof of Theorem \ref{38} follows straightforwardly from the 3-D case with some minor modifications, and we omit the proof here.

To investigate the  formation of singularities, we introduce the following definition about the ``isolated mass group''(see Xin-Yan\cite{Xin2013On})
\begin{definition}{\rm (Isolated mass group)} $\left(\rho_0, u_0\right)$ is said to have an isolated mass group $\left(A^0, B^0\right)$, if there are two bounded open intervals $A^0=(a_0,b_0)\subset [a_0,b_0] \subset B^0 \subseteq (-d,d)  \subset \mathbb{R}$ satisfying
\begin{equation}\label{86}
\displaystyle \rho_0(x)=0 	\quad \text{for} \ x \in B^0 \backslash A^0, \quad \int_{A^0} \rho_0(x) \mathrm{d} x>0;\quad \left.u_0(x)\right|_{x=a_0, b_0}=\bar{u}_0
\end{equation}
for some real constants $a_0, b_0, \bar{u}_0 \in \mathbb{R}$,   and $d>0$.
\end{definition}

Then the following theorem shows that an isolated mass group in initial data is a sufficient condition for the finite time singularity formation of the regular solutions.
\begin{theorem}\label{77}
{\rm (Blow-up by isolated mass group)} Let \eqref{78} hold. Suppose that initial data $(\rho_0(x),u_0(x))$ satisfy \eqref{66} and 
\begin{equation}\label{98}
  \rho_0\in L^1(\mathbb{R}),
\end{equation} 
and  have an isolated mass group $\left(A^0, B^0\right)$. Let $T_m$ be the maximal existence time of the regular solution $(\rho(t, x), u(t, x))$ obtained in Theorem \ref{38}. Then $(\rho(t, x), u(t, x))$ blows up in finite time, i.e., $T_m<+\infty$.
\end{theorem}

\begin{remark}
Compared with the blow-up result obtained  in \cite{clsxnk} for isentropic inviscid fluid, Theorem \ref{77} shows that even with the viscous effect, the appearance of vacuum still leads to finite time blow-up of regular solution, which implies that the viscosity is not strong enough to prevent the formation of singularity due to vacuum.
\end{remark}

\begin{remark}
  One can find the following class of initial data $(\rho_0,u_0)$ satisfying \eqref{66}-\eqref{98}:  
  \begin{equation*}
    \rho_0(x)=\left\{
    \begin{aligned}
    &0   &c\le |x| \le c+1,\\
    &(1+|x|)^{-k_1}   &else,
    \end{aligned}
    \right.
    \qquad  \quad
    u_0(x)=(1+|x|)^{-k_2}, 
  \end{equation*}
where constants $c>0$, $k_1>\max\big\{1,\frac{1}{\gamma-1},\frac{1}{\delta-1}\big\}$, $k_2>\frac12.$
\end{remark}
The rest of this paper is organized as follows. \S \ref{97} is
dedicated for the preliminary lemmas to be used later.  \S \ref{108} is devoted to proving the uniform $L^\infty$ bound of the mass density, see Lemma \ref{60}; and Theorem  \ref{77} is proved in \S \ref{110}.

\section{Preliminaries} \label{97}
This section will be devoted to present some necessary preliminary lemmas to be used later, the first one is some compactness results from the Aubin-Lions Lemma.
\begin{lemmaNoParens}
 [\cite{Simon1987compact}]\label{22}  Let $X_0, X$ and $X_1$ be three Banach spaces with $X_0 \subset X \subset X_1$. Suppose that $X_0$ is compactly embedded in $X$ and that $X$ is continuously embedded in $X_1$. Then
\begin{itemize}
\item Let $G$ be bounded in $L^p\left([0, T]; X_0\right)$ where $1 \leq p<\infty$, and $\frac{\partial G}{\partial t}$ be bounded in $L^1\left([0, T]; X_1\right)$, then $G$ is relatively compact in $L^p([0, T]; X)$.
\item Let $G$ be bounded in $L^{\infty}\left([0, T]; X_0\right)$ and $\frac{\partial G}{\partial t}$ be bounded in $L^p\left([0, T]; X_1\right)$ with $p>1$, then $G$ is relatively compact in $C([0, T]; X)$.
\end{itemize}
\end{lemmaNoParens}

The second one is Fatou's lemma.
\begin{lemmaNoParens}\label{89}
 Given a measure space $(X, \mathcal{F}, \nu)$ and a set $E  \in \mathcal{F}$, let $\left\{f_n\right\}$ be a sequence of $\left(\mathcal{F}, \mathcal{B}_{\mathbb{R}_{\geq 0}}\right)$-measurable nonnegative functions $f_n: E  \rightarrow[0, +\infty]$. Define the function $f: E  \rightarrow[0, +\infty]$ by setting
$$
f(s)=\liminf _{n \rightarrow \infty} f_n(s)
$$
for every $s \in E $. Then $f$ is $\left(\mathcal{F}, \mathcal{B}_{\mathbb{R}_{\geq 0}}\right)$-measurable, and
$$
\int_{E } f(s) \mathrm{d} \nu \leq \liminf _{n \rightarrow \infty} \int_{E } f_n(s) \mathrm{d} \nu .
$$
\end{lemmaNoParens}

The third lemma is Jensen's inequality.
\begin{lemmaNoParens}[\cite{Hayk2018Alge}]\label{90}
Given a measure space $(X, \mathcal{F}, \nu)$ and $E \in \mathcal{F}$, $\nu(E)<+\infty.$ Let $g:E\rightarrow (a,b)$ is an integrable function and $F$ is a convex function on $(a,b)$, then
\begin{equation*}
  F\Big(\frac{1}{\nu(E)}\int_E g(s)\mathrm{d}\nu\Big)\leq \frac{1}{\nu(E)}\int_E F\Big(g(s)\Big) \mathrm{d}\nu.
\end{equation*} 
\end{lemmaNoParens}

\section{The uniform $L^\infty$ bound of the density}\label{108}
 This section aims at proving the uniform $L^\infty$ bound of the mass density, i.e., Lemma \ref{60}. Let $T>0$ and $(\rho, u)$ be regular solutions to Cauchy problem \eqref{1}-\eqref{44} in $[0, T] \times \mathbb{R}$ obtained in Theorem \ref{38}. Hereinafter, we denote $C_0$ (resp., $C$) a generic positive constant depending only on $\left(\rho_0, u_0, \gamma,  \delta\right)$ (resp., $\left(C_0, T\right)$), which may be different from line to line.



\begin{lemma}\label{84} Let $(\rho, u)$ in $\left[0, T\right] \times \mathbb{R}$ be the regular solution to Cauchy problem \eqref{1}-\eqref{44} obtained in  Theorem \ref{38}. Then
\begin{equation}\label{79}
\rho \in C\left(\left[0, T\right] ; L^1\right) \quad \text { if } \quad \rho_0 \in L^1 \quad \text { additionally. }
\end{equation}
Moreover, it holds that 
\begin{equation}\label{111}
  \int_{\mathbb{R}} \rho  \mathrm{d}x = \int_{\mathbb{R}}\rho_0 \mathrm{d}x.
\end{equation}
\end{lemma}

\begin{proof} First, one proves that
\begin{equation}\label{83}
\rho \in L^{\infty}\left(\left[0, T\right] ; L^1\right) .
\end{equation}
For this purpose, let $f: \mathbb{R}^{+} \rightarrow \mathbb{R}$ satisfy
$$
f(s)= \begin{cases}1 & s \in\left[0, \frac{1}{2}\right], \\ \text { nonnegative polynomial, } & s \in\left[\frac{1}{2}, 1\right], \\ e^{-s}, & s \in[1, \infty),\end{cases}
$$
such that $f \in C^2$. Then there exists a generic constant $C_*>0$ such that $\left|f^{\prime}(s)\right| \leq$ $C_* f(s)$. Define, for any $R>1, f_R(x)=f\big(\frac{|x|}{R}\big)$. Then, by the definition of $f(x)$ and $f_R(x)$, we can obtain that
$$
|f^{'}_R(x)|\leq C|f^{'}(|x|)|\leq Cf(|x|)\leq C f_R(x),
$$
for some positive constant $C$ depending only on $C_*$ but independent of $R$.

According to Theorem \ref{38}, one can obtain that for any given $R>1$,
$$
\int_{\mathbb{R}}\big( ( |\rho_t|+|\rho_x u |+|\rho u_x| ) f_R(x)+\rho f_R(x)+|\rho u  f_R^{\prime}(x)| \big) \mathrm{d} x<\infty .
$$

Multiplying $\eqref{1}_1$ by $f_R(x)$ and integrating over $\mathbb{R}$, one has
$$
\begin{aligned}
\frac{\mathrm{d}}{\mathrm{d} t} \int_{\mathbb{R}} \rho f_R(x) \mathrm{d} x & =-\int_{\mathbb{R}}\left(\rho_x u+\rho u_x\right) f_R(x) \mathrm{d} x \\
& =\int_{\mathbb{R}} \rho u f_R^{\prime}(x) \mathrm{d} x \leq C|u|_{\infty} \int_{\mathbb{R}} \rho f_R(x) \mathrm{d} x,
\end{aligned}
$$
which, along with the Gronwall inequality, implies that
\begin{equation}\label{91}
\int_{\mathbb{R}} \rho f_R(x) \mathrm{d} x \leq C \quad \text { for } \quad t \in\left[0, T\right] .
\end{equation}

Since $\rho f_R(x) \rightarrow \rho$ almost everywhere as $R \rightarrow \infty$, then it follows from Lemma \ref{89} and \eqref{91} that
$$
\int_{\mathbb{R}} \rho \mathrm{d} x \leq \liminf _{R \rightarrow \infty} \int_{\mathbb{R}} \rho f_R(x) \mathrm{d} x \leq C \quad \text { for } \quad t \in\left[0, T\right] .
$$

Second, from \eqref{78}, without loss of generality, one can assume that $\gamma\leq 3$ and  then
\begin{equation}\label{87}
  \begin{aligned}
  |\rho|_2&=  |\rho^{\frac{\gamma-1}{2}}\rho ^{\frac{3-\gamma}{2}}|_2\leq |\rho|_\infty^{\frac{3-\gamma}{2}}|\rho^{\frac{\gamma-1}{2}}|_2\leq C,\\
  |\rho_x|_2&=|\rho^{\frac{\gamma}{2}-\frac{3}{2}}\rho_x\rho^{\frac{3}{2}-\frac{\gamma}{2}}|_2
  \leq |\rho|_\infty^{\frac{3}{2}-\frac{\gamma}{2}}|\rho^{\frac{\gamma}{2}-\frac{3}{2}}\rho_x|_2\leq C.
  \end{aligned}
\end{equation} 

Thus, according to $\eqref{1}_1$, \eqref{87}, Theorem \ref{38} and H\"older's inequality, one has
$$
\left|\rho_t\right|_1 \leq   C\left(|\rho_x|_2|u|_2+|\rho|_2\left|u_x\right|_2\right) \leq C,
$$
which, along with \eqref{83} and Lemma \ref{22}, implies that \eqref{79}. 

According to Theorem \ref{38} and \eqref{87}, one has
\begin{equation*}
|\rho u |_1+| (\rho u)_x|_1 \leq C(|\rho|_2|u|_2+|\rho_x|_2|u|_2+|\rho|_2|u_x|_2)\leq C,
\end{equation*}
which means that
\begin{equation}\label{92}
  \rho u(t,\cdot) \in W^{1,1}(\mathbb{R}).
\end{equation}

It follows from $\eqref{1}_1$, \eqref{44}, \eqref{79} and \eqref{92} that 
\begin{equation*}
\frac{\mathrm{d}}{\mathrm{d}t} \int_{\mathbb{R}} \rho  \mathrm{d}x=-\int_{\mathbb{R}} (\rho u )_x \mathrm{d}x=0,
\end{equation*}
via integrating over $[0,t]$, one can get \eqref{111}.  The proof of Lemma \ref{84} is complete. \qed 
\end{proof}

\begin{lemma}\label{60} For any $T>0$, it holds that
\begin{equation}\label{94}
|\rho(t, \cdot)|_{\infty} \leq C_0 \quad \text { for } \quad 0 \leq t \leq T.
\end{equation}
\end{lemma}
\begin{proof}
First, integrating $\eqref{1}_2$ over $(-\infty, x]$ with respect to $x$, one has
\begin{equation}\label{61}
\frac{\mathrm{d}}{\mathrm{d} t} \int_{-\infty}^x \rho u(t, z) \mathrm{d} z+\rho u^2+ \rho^\gamma= \rho^\delta u_x,
\end{equation}
where one has used the far field behavior $\eqref{44}$.
Denote $\xi(t, x)=\int_{-\infty}^x \rho u(t, z) \mathrm{d} z$. Then it follows from $\eqref{1}_1 $  and \eqref{61} that
\begin{equation}\label{62}
\xi_t+u \xi_x+ \rho^{\delta-1}\left(\rho_t+u \rho_x\right)+ \rho^\gamma=0.
\end{equation}

Second, let $y(t ; x)$ be the solution of the problem
\begin{equation*}
\left\{\begin{array}{l}
\frac{\mathrm{d} y(t ; x)}{\mathrm{d} t}=u(t, y(t ; x)), \\
y(0 ; x)=x.
\end{array}\right.
\end{equation*}
It follows from \eqref{62} that
\begin{equation*}
\frac{\mathrm{d}}{\mathrm{d} t}\left(\xi+\frac{1}{\delta} \rho^\delta\right)(t, y(t ; x)) \leq 0,
\end{equation*}
which implies
\begin{equation}\label{64}
\left(\xi+\frac{1}{\delta} \rho^\delta\right)(t, y(t ; x)) \leq \xi(0, y(0 ; x))+\frac{1}{\delta} \rho_0^\delta(x).
\end{equation}

To get the uniform $L^\infty$ bound of $\rho,$ now we need to estimate $\xi(t,y(t;x))$. It follows from $\eqref{1}_1$ that
\begin{equation}\label{75}
\frac{\rho^{\gamma}_t}{\gamma-1}+\frac{(\rho^{\gamma} u)_x}{\gamma-1}+\rho^{\gamma} u_x=0 .
\end{equation}
Then multiplying $\eqref{1}_2$ by $u$, adding the resulting equation to \eqref{75}, and integrating over $\mathbb{R}$, one arrives at
\begin{equation}\label{76}
\frac{\mathrm{d}}{\mathrm{d} t} \int_{\mathbb{R}}\left(\frac{1}{2} \rho u^2+\frac{\rho^{\gamma}}{\gamma-1}\right) \mathrm{d} x+ \int_{\mathbb{R}} \rho^\delta u_x^2 \mathrm{~d} x=0,
\end{equation}
where the far field behavior $\eqref{44}$ has been used. From \eqref{98}, one has 
\begin{equation}
\begin{aligned}
  &\int_{\mathbb{R}} \Big(\frac{1}{2} \rho_0 u_0^2+\frac{\rho_0^\gamma}{\gamma-1}\Big) \mathrm{d}x\leq C|\rho_0|_\infty|u_0|_2^2+C|\rho_0|_\infty \int_{\mathbb{R}} \rho_0^{\gamma-1}\mathrm{d}x\leq C_0.\\
\end{aligned}
\end{equation}
Then, integrating \eqref{76} over $[0, t]$, one can get that
\begin{equation}\label{112}
\int_{\mathbb{R}}\left(\frac{1}{2} \rho u^2+\frac{\rho^{\gamma}}{\gamma-1} \right)  \mathrm{d} x+ \int_0^t \int_{\mathbb{R}} \rho^\delta u_x^2 \mathrm{~d} x \mathrm{~d} s \leq C_0 \quad \text { for } \quad 0 \leq t \leq T.
\end{equation}
From \eqref{111} and \eqref{112}, we have 
\begin{equation}\label{63}
\xi(t, y(t ; x)) \leq|\sqrt{\rho} u|_2|\rho|_1^{\frac{1}{2}} \leq C_0 \quad \text { for } \quad 0 \leq t \leq T.
\end{equation} 

Then it follows from  \eqref{64} and \eqref{63} that
\begin{equation*}
|\rho^\delta(t, \cdot)|_{\infty} \leq C_0 \quad \text { for } \quad 0 \leq t \leq T.
\end{equation*}
The proof of Lemma \ref{60} is complete. \qed
\end{proof}
\section{Singularity Formation}\label{110}
Some necessary definitions are provided in order to demonstrate Theorem \ref{77}.
\begin{definition} (Particle path and flow map) Let $x\left(t ; x_0\right)$ be the particle path starting from $x_0$ at $t=0$, i.e.,
\begin{equation}\label{41}
\frac{\mathrm{d}}{\mathrm{d} t} x\left(t ; x_0\right)=u(t, x(t ; x_0)), \quad x\left(0 ; x_0\right)=x_0 .
\end{equation}
Let $A(t), B(t), B(t) \backslash A(t)$ be the images of $A^0, B^0$, and $B^0 \backslash A^0$, respectively, under the flow map of \eqref{41}, i.e.,
\begin{equation*}
\begin{aligned}
A(t) & =\left\{x\left(t ; x_0\right) \mid x_0 \in A^0\right\}\triangleq(a(t),b(t)), \\
B(t) & =\left\{x\left(t ; x_0\right) \mid x_0 \in B^0\right\}, \\
B(t) \backslash A(t) & =\left\{x\left(t ; x_0\right) \mid x_0 \in B^0 \backslash A^0\right\} .
\end{aligned}
\end{equation*}
\end{definition}

In the rest part of this paper, we will use the following useful physical quantities on the fluids in $A(t)$:
\begin{equation*}
\begin{aligned}
m(t) & =\int_{A(t)} \rho(t,x) \mathrm{d} x \quad \text { (total mass), } \\
M(t) & =\int_{A(t)} \rho(t,x)x^2 \mathrm{~d} x \quad \text { (second moment), } \\
F(t) & =\int_{A(t)} \rho(t,x) u(t,x)  x \mathrm{~d} x \quad \text { (radial component of momentum), } \\
\varepsilon(t) & =\int_{A(t)}\left(\frac{1}{2} \rho u^2+\frac{p}{\gamma-1}\right)(t,x) \mathrm{d} x \quad \text { (total energy). }
\end{aligned}
\end{equation*}

The following lemma confirms the invariance of the volume $|A(t)|$ for regular solutions.

\begin{lemma}\label{53} Suppose that the initial data $\left(\rho_0, u_0\right)$ have an isolated mass group $\left(A^0, B^0\right)$. Then for the regular solution $(\rho, u)$ on $ \left[0, T_m\right)\times \mathbb{R}$ to Cauchy problem \eqref{1}-\eqref{44}, we have
\begin{equation*}
|A(t)|=\left|A^0\right|, \quad t \in\left[0, T_m\right) .
\end{equation*}
\end{lemma}
\begin{proof} Since
\begin{equation*}
\rho\left(x\left(t ; x_0\right), t\right)=\rho_0\left(x_0\right) \exp \left(\int_0^t  u_x\left(s, x\left(s ; x_0\right) \right) \mathrm{d} s\right)
\end{equation*}
for any $x_0\in \mathbb{R}$, thus
\begin{equation}\label{85}
\rho \equiv 0 \quad \text { in } \quad B(t) \backslash A(t).
\end{equation}
According to the definition of regular solutions, one has
\begin{equation*}
u_t+u u_x=0 \quad \text { in } \quad B(t) \backslash A(t).
\end{equation*}
Therefore, \eqref{85} shows that $u$ is invariant along the particle path $x\left(t ; x_0\right)$ with $x_0 \in B^0 \backslash A^0$.

For any $A^0=(a_0,b_0)$, define
\begin{equation*}
\begin{aligned}
&\frac{\mathrm{d}}{\mathrm{d} t} x^1\left(t ;a_0\right)=u\left(t,x^1\left(t ; a_0\right) \right), \quad x^1\left(0 ; a_0\right)=a_0;\\
&\frac{\mathrm{d}}{\mathrm{d} t} x^2\left(t ;b_0\right)=u\left(t,x^2\left(t ; b_0\right) \right), \quad x^2\left(0 ; b_0\right)=b_0.\\
\end{aligned}
\end{equation*}
Thus
\begin{equation*}
\begin{split}
&\frac{\mathrm{d}}{\mathrm{d} t}\left(x^1\left(t ; a_0\right)-x^2\left(t ;b_0\right)\right)
=u\left(t,x^1\left(t ;a_0\right) \right)-u\left(t, x^2\left(t ; b_0\right) \right)
=\bar{u}_0-\bar{u}_0=0,
\end{split}
\end{equation*}
which implies that
\begin{equation*}
|A(t)|=\left|A^0\right|, \quad t \in\left[0, T_m\right).  
\end{equation*} \qed
\end{proof}

We point out that, although the volume of $A(t)$ is invariant, the vacuum boundary $\partial A(t)$ travels with constant velocity $\bar{u}_0$. The following well-known Reynolds transport theorem (c.f. Appendix C of \cite{Evans2010Partial}) is useful.
\begin{lemma}\label{43} For any $G(t,x) \in C^1\left( \mathbb{R}^{+}\times \mathbb{R}^d\right)$, one has
\begin{equation*}
\frac{\mathrm{d}}{\mathrm{d} t} \int_{A(t)} G(t,x) \mathrm{d} x=\int_{A(t)} G_t(t,x) \mathrm{d} x+\int_{\partial A(t)}G(t,x)(u(t,x)\cdot \vec{n})\mathrm{d}S,
\end{equation*}
where $\vec{n}$ is the outward unit normal vector to $\partial A(t)$ and $u$ is the velocity of the fluid.
\end{lemma}

The mass is conserved according to the following lemma.
\begin{lemma}\label{54} Suppose that the initial data $\left(\rho_0, u_0\right)(x)$ have an isolated mass group $\left(A^0, B^0\right)$, then for the regular solution $(\rho, u)(t, x)$ on $\left[0, T_m\right) \times \mathbb{R}$ to Cauchy problem \eqref{1}-\eqref{44}, we have
\begin{equation*}
m(t)=m(0)\quad \text { for } \ \ t \in\left[0, T_m\right).
\end{equation*}
\end{lemma}
\begin{proof}
From $\eqref{1}_1$, $\eqref{85}$ and Lemma \ref{43}, direct computation shows that
\begin{equation*}
\begin{aligned}
\frac{\mathrm{d}}{\mathrm{d} t} m(t) & =\int_{a(t)}^{b(t)} \rho_t \mathrm{~d} x+ \rho(t,b(t))u(t,b(t))- \rho(t,a(t))u(t,a(t)) \\
& =\int_{a(t)}^{b(t)}-(\rho u)_x \mathrm{d} x=-\rho u|_{a(t)}^{b(t)}=0,
\end{aligned}
\end{equation*}
which implies that $m(t)=m(0).$ \qed
\end{proof}
Motivated by \cite{Xin1998Blow}, we define the functional
\begin{equation*}
\begin{aligned}
I(t) & =M(t)-2(t+1) F(t)+2(t+1)^2 \varepsilon(t) \\
& =\int_{A(t)}\big(x-(t+1) u\big)^2 \rho \mathrm{d} x+\frac{2}{\gamma-1}(t+1)^2 \int_{A(t)} p \mathrm{~d} x .
\end{aligned}
\end{equation*}

We now follow the arguments of \cite{Xin1998Blow} to prove Theorem \ref{77}, where a key observation is that the viscous force $\mu(\rho)u_x=\rho^{\delta}u_x=0$ in the vacuum region for the degenerate fluid.

\noindent \textbf{Proof of Theorem \ref{77}.}
First, it is easy to find that
\begin{equation}\label{bdc}
\rho^{\delta}u_x|_{\partial A(t)}=0.
\end{equation}

Multiplying $\eqref{1}_2$ by $u$, one has 
\begin{equation}\label{95}
\rho u u_t+\rho u^2u_x+up_x=\rho^{\delta}uu_x.
\end{equation}
From \eqref{95} and using the fact that
$$
\frac{1}{2}(\rho u^2)_t=\rho u u_t-\frac{1}{2}u^3\rho_x-\frac{1}{2}\rho u^2u_x,
$$
we obtain that 
\begin{equation}\label{96}
\frac{1}{2}(\rho u^2)_t+\frac{1}{2}(\rho u^3)_x+up_x=u(\rho^{\delta} u_x)_x.
\end{equation}
Then, adding \eqref{75} and \eqref{96}, one has 
\begin{equation}\label{45}
\left(\frac{1}{2} \rho u^2+\frac{p}{\gamma-1}\right)_t=-\left(\frac{1}{2} \rho u^3 \right)_x-\frac{\gamma}{\gamma-1} (p u)_x+u (\rho^{\delta}u_x)_x.
\end{equation}
Via $\eqref{1}_1$, $\eqref{1}_2$,   \eqref{45}, Lemma \ref{43}, and integrating by parts, one has
\begin{equation}\label{50}
\begin{aligned}
\frac{\mathrm{d}}{\mathrm{d} t} I(t)  =&\frac{\mathrm{d}}{\mathrm{d} t} M(t)-2(t+1) \frac{\mathrm{d}}{\mathrm{d} t} F(t)+2(t+1)^2 \frac{\mathrm{d}}{\mathrm{d} t} \varepsilon(t)-2 F(t)+4(t+1) \varepsilon(t) \\
 =&\frac{2}{\gamma-1}(3-\gamma)(t+1) \int_{A(t)} p \mathrm{~d} x-2(t+1)\int_{A(t)} x (\rho^{\delta}u_x)_x \mathrm{d}x\\
&+2(t+1)^2\int_{A(t)} u (\rho^{\delta}u_x)_x \mathrm{d}x\\
=&\frac{2}{\gamma-1}(3-\gamma)(t+1) \int_{A(t)} p \mathrm{~d} x-2(t+1)\int_{A(t)} \big((x \rho^{\delta} u_x)_x-\rho^\delta u_x\big) \mathrm{d}x\\
&+2(t+1)^2\int_{A(t)}\big((\rho^{\delta} u u_x)_x-\rho^\delta(u_x)^2\big) \mathrm{d}x\\
=&\frac{2}{\gamma-1}(3-\gamma)(t+1) \int_{A(t)} p \mathrm{~d} x\\
&+2(t+1)\int_{A(t)} \rho^\delta u_x \mathrm{d}x-2(t+1)^2\int_{A(t)} \rho^\delta(u_x)^2 \mathrm{d}x,
\end{aligned}
\end{equation}
where we have used \eqref{bdc}.
By Cauchy's inequality,
\begin{equation}\label{51}
\begin{aligned}
&2(t+1)\int_{A(t)} \rho^\delta u_x \mathrm{d}x-2(t+1)^2\int_{A(t)} \rho^\delta(u_x)^2 \mathrm{d}x\\
\leq & -2(t+1)^2\int_{A(t)} \rho^\delta(u_x)^2 \mathrm{d}x +2(t+1)^2\int_{A(t)} \rho^\delta(u_x)^2 \mathrm{d}x+2\int_{A(t)} \rho^\delta \mathrm{d}x\\
=&2\int_{A(t)} \rho^\delta \mathrm{d}x\leq  M,
\end{aligned}
\end{equation}
the last inequality is due to  Lemma \ref{60}, Lemma \ref{53} and the boundness of $A^0$, where constant $M=M(A^0, C_0)>0$.

It follows from \eqref{50}-\eqref{51} that
\begin{equation}
\frac{\mathrm{d}}{\mathrm{d} t} I(t) \leq \frac{2(3-\gamma)}{\gamma-1}(t+1) \int_{A(t)} p \mathrm{~d} x+M.
\end{equation}
If $3-\gamma\leq 0,$ i.e. $\gamma \geq 3,$ then 
\begin{equation*}
\frac{\mathrm{d}}{\mathrm{d} t} I(t) \leq M,
\end{equation*}
which implies that
\begin{equation}\label{56}
I(t)\leq Mt+C.
\end{equation}
If $3-\gamma> 0,$ i.e. $ 1<\gamma <3,$ then it follows from definition of  $I(t)$ that
\begin{equation}\label{52}
\frac{\mathrm{d}}{\mathrm{d} t} I(t)\leq \frac{a}{t+1} I(t) +M,
\end{equation}
where $a=3-\gamma \in(1,2)$.
Multiplying $(1+t)^{-a}$ on both sides of \eqref{52}, one has
\begin{equation}\label{88}
  \frac{\mathrm{d}}{\mathrm{d}t}\left(I(t)(1+t)^{-a} \right)\leq M (1+t)^{-a}.
\end{equation}
Integrating \eqref{88} from 0 to $t$ leads to
\begin{equation}
\begin{aligned}
I(t)&\leq \Big( I(0)+\frac{M}{a-1}\Big)(t+1)^a\le C(t+1)^a .
\end{aligned}
\end{equation}
On the other hand, from the definition of $I(t)$, Lemma \ref{90} and Lemma \ref{53}, it holds that
\begin{equation}\label{55}
\begin{aligned}
I(t) & \geq \frac{2(t+1)^2}{\gamma-1}\left|A^0\right| \int_{A(t)}  \rho^\gamma(x, t) \frac{\mathrm{d} x}{|A(t)|} \\
&\geq \frac{2(t+1)^2}{\gamma-1}\Big(\frac{1}{|A(t)|}\int_{A(t)}\rho \mathrm{d}x\Big)^\gamma\\
& \geq \frac{2(t+1)^2}{\gamma-1}\left|A^0\right|^{1-\gamma} m(0)^\gamma \geq C_0(1+t)^2,
\end{aligned}
\end{equation}
where we used the fact in Lemma \ref{54} that
\begin{equation*}
m(t)=\int_{A(t)} \rho(x, t) \mathrm{d} x=\int_{A^0} \rho_0(x) \mathrm{d} x=m(0) .
\end{equation*}
Then $T_m<+\infty$ follows immediately, otherwise a contradiction forms between \eqref{55} and \eqref{56} or \eqref{52}.

\section{Acknowledgments}
 The research was supported in part by  National Natural Science Foundation of China under Grants 12371221, 12161141004 and 11831011.

The authors declare that they have no conflict of interest.

\end{document}